\documentclass[twoside,tbtags,leqno]{amsart}  

\usepackage{amsmath}

\numberwithin{equation}{section} 

\newtheorem{theorem}{Theorem}[section]
\newtheorem{corollary}[theorem]{Corollary}
\newtheorem{lemma}[theorem]{Lemma}
\newtheorem{proposition}[theorem]{Proposition}

\theoremstyle{definition}
\newtheorem{Remark}[theorem]{Remark}

\newcommand{\lc}{\left\lceil}
\newcommand{\lf}{\left\lfloor}
\newcommand{\rc}{\right\rceil}
\newcommand{\rf}{\right\rfloor}

\newcommand{\EE}{{\bf  E}}

\newcommand{\ZZ}{{\mathbb Z}}

\newcommand{\PP}{{\bf  P}}

\newcommand{\Var}{{\bf Var}}

\newcommand{\ct}{{\tilde{c}}}

\newcommand{\Lc}{{\mathcal L}}

\newcommand{\Leq}{{\,\stackrel{\Lc}{=}\,}}

\newcommand{\begp}{\begin{proposition}}
\newcommand{\enp}{\end{proposition}}
\newcommand{\begt}{\begin{theorem}}
\newcommand{\ent}{\end{theorem}}
\newcommand{\begl}{\begin{lemma}}
\newcommand{\enl}{\end{lemma}}
\newcommand{\begc}{\begin{corollary}}
\newcommand{\enc}{\end{corollary}}
\newcommand{\begcl}{\begin{claim}}
\newcommand{\encl}{\end{claim}}
\newcommand{\begr}{\begin{Remark}\rm}
\newcommand{\enr}{\end{Remark}}
\newcommand{\begal}{\begin{algorithm}}
\newcommand{\enal}{\end{algorithm}}
\newcommand{\begd}{\begin{definition}}
\newcommand{\enf}{\end{definition}}
\newcommand{\begx}{\begin{example}}
\newcommand{\enx}{\end{example}}
\newcommand{\bega}{\begin{array}}
\newcommand{\ena}{\end{array}}

\newcommand{\sfrac}[2]{{\textstyle\frac{#1}{#2}}}

\def\rompar(#1){\textup(#1\textup)}    

\newcommand\bigpar[1]{\bigl(#1\bigr)}
\newcommand\bigbrack[1]{\bigl[#1\bigr]}
\newcommand\Bigpar[1]{\Bigl(#1\Bigr)}

\newcommand\lrpar[1]{\left(#1\right)}
\newcommand\lrbrack[1]{\left[#1 \right]}

\newcommand\ga{\alpha}

\newcommand\gl{\lambda}
\newcommand\gL{\Lambda}

\newcommand{\refS}[1]{Section~\ref{#1}}
\newcommand{\refT}[1]{Theorem~\ref{#1}}

\newcommand{\refL}[1]{Lemma~\ref{#1}}
\newcommand{\refP}[1]{Proposition~\ref{#1}}
\newcommand{\refR}[1]{Remark~\ref{#1}}

\newcommand\ie{i.e.\spacefactor=1000}
\newcommand\eg{e.g.\spacefactor=1000}

\newcommand\cf{{cf.}\spacefactor=1000}

\renewcommand\Re{\operatorname{Re}}
\newcommand\bbC{\mathbb C}

\newcommand\ils{I(\gl;s)}
\newcommand\ilsi{I(\gl;s-1)}

\newcommand\Holder{H\"older}

\newcommand\fall[1]{^{\underline{#1}}}

\newcommand{\unif}{\mathrm{unif}}

\newcommand\urladdrx[1]{{\urladdr{\def~{{\tiny$\sim$}}#1}}}

\begin{document}

\newcommand{\tab}[0]{\hspace{.1in}}

\title[The Number of Bit Comparisons Used by Quicksort]
{The Number of Bit Comparisons Used by Quicksort:\\
An Average-case Analysis}

\author{James Allen Fill}
\address{Department of Applied Mathematics and Statistics,
The Johns Hopkins University,
3400 N.~Charles Street,
Baltimore, MD 21218-2682 USA}
\email{jimfill@jhu.edu}
\urladdrx{http://www.ams.jhu.edu/~fill/}
\thanks{Research of the first author supported by NSF grants DMS-0104167 and DMS-0406104
and by the Acheson~J.~Duncan Fund for the Advancement of Research in
Statistics.}
\author{Svante Janson}
\address{Department of Mathematics,
Uppsala University, 
P.~O.~Box 480,
SE-751 06 Uppsala, Sweden}
\email{svante.janson@math.uu.se}
\urladdrx{http://www.math.uu.se/~svante/}

\date{February~10, 2012}

\begin{abstract} \small\baselineskip=9pt
The analyses of many algorithms and data structures (such as digital search
trees) for searching and sorting are based on the representation of the keys
involved as bit strings and so count the number of bit comparisons.  On the
other hand, the standard analyses of many other algorithms (such as {\tt
Quicksort}) are performed in terms of the number of key comparisons.  We
introduce the prospect of a fair comparison between algorithms of the two
types by providing an average-case analysis of the number of bit comparisons
required by {\tt Quicksort}.  Counting bit comparisons rather than key
comparisons introduces an extra logarithmic factor to the asymptotic average
total.  We also provide a new algorithm,
``{\tt BitsQuick}'',
that reduces
this factor to constant order by eliminating needless bit comparisons.
\end{abstract}

\maketitle

\section{Introduction and summary}
\label{S:intro}

Algorithms for sorting and searching (together with their accompanying
analyses) generally fall into one of two categories: either the algorithm is
regarded as comparing items pairwise irrespective of their internal structure
(and so the analysis focuses on the number of comparisons), or else it is
recognized that the items (typically numbers) are represented as bit
strings and that the algorithm operates on the individual bits.  Typical
examples of the two types are {\tt Quicksort} and digital search trees,
respectively; see~\cite{MR56:4281}.

In this paper---a substantial expansion of the extended abstract~\cite{fjbits2004}---we 
take a first step towards bridging the gap
between the two points of view, in order to facilitate run-time comparisons
across the gap, by answering the following question posed many years ago by
Bob Sedgewick [personal communication]: What is the bit complexity of {\tt
Quicksort}?  (For a discussion of related work that has transpired in the time between~\cite{fjbits2004}
and this paper, see~\refR{R:lit} at the end of this section.)

More precisely, we consider {\tt Quicksort} (see \refS{S:fixkey} for a
review) applied to~$n$ distinct keys (numbers) from the interval $(0, 1)$. 
Many authors (Knuth~\cite{MR56:4281}, R\'{e}gnier~\cite{MR90k:68132},
R\"{o}sler~\cite{MR92f:68028}, Knessl and Szpankowski~\cite{MR2000b:68051},
Fill and Janson~\cite{FillJansonbook}~\cite{MR1932675}, Neininger and
Ruschendorff~\cite{NR}, and others) have studied $K_n$, the (random) number of
key comparisons performed by the algorithm.  This is a natural measure of the
cost (run-time) of the algorithm, if each comparison has the same cost.  On
the other hand, if comparisons are done by scanning the bit representations of
the numbers, comparing their bits one by one, then the cost of comparing two
keys is determined by the number of bits compared until a difference is
found.  We call this number the number of \emph{bit comparisons} for the key
comparison, and let $B_n$ denote the total number of bit comparisons when~$n$
keys are sorted by {\tt Quicksort}.

We assume that the keys $X_1, \dots, X_n$ to be sorted are independent
random variables with a common continuous distribution~$F$ over $(0, 1)$.  It
is well known that the distribution of the number~$K_n$ of key comparisons does
not depend on~$F$.  This invariance clearly fails to extend to the
number~$B_n$ of bit comparisons, and so we need to specify~$F$.

For simplicity, we study mainly the case that~$F$ is the uniform
distribution, and, throughout, the reader should assume 
this as
the default.  
But we also give a result valid for a general absolutely
continuous distribution~$F$ over $(0, 1)$ (subject to a mild integrability
condition on the density). 

In this paper we focus on the mean of~$B_n$.  One of our main
results is the following \refT{T:fixunif}, the concise version of
which is the asymptotic equivalence
$$
\EE\,B_n \sim n (\ln n) (\lg n)\mbox{\ as $n \to \infty$}.
$$
Throughout, we use $\ln$ (respectively, $\lg$) to denote natural (resp.,\
binary) logarithm, and use $\log$ when the base doesn't matter (for example,
in remainder estimates).  The symbol~$\doteq$ is used to denote approximate
equality, and~$\gamma \doteq 0.57722$ is Euler's constant.

\begin{theorem}\label{T:fixunif}
If the keys $X_1, \dots, X_n$ are independent and uniformly distributed on
$(0, 1)$, then the number~$B_n$ of bit comparisons required to sort these
keys using {\tt Quicksort} has expectation given by the following exact and
asymptotic expressions: 
\begin{align}
\EE\,B_n
\label{fixexact1}
 &= 2 \sum_{k = 2}^n (-1)^k \binom{n}{k} \frac{1}{(k - 1) k [1 - 2^{- (k -
       1)}]} \\
\label{fixasy}
 &= n (\ln n) (\lg n) - c_1 n \ln n + c_2 n + \pi_n n + O(\log n),
\end{align}
where, with $\beta := 2 \pi / \ln 2$,
\begin{align*}
c_1   &  :=   \frac{1}{\ln 2} (4 - 2 \gamma - \ln 2) \doteq 3.105, \\
c_2   &  :=   \frac{1}{\ln 2} \left[\frac{1}{6} (6 - \ln 2)^2 - (4 - \ln 2)
                 \gamma + \frac{\pi^2}{6} + \gamma^2 \right] 
\doteq 6.872, 
\end{align*}
and
\begin{equation}
  \label{pin}
\pi_n := \sum_{k \in \ZZ:\,k \neq 0} \frac{i}{\pi k (-1 - i\beta k)} \Gamma(-1
- i \beta k) n^{i \beta k} 
\end{equation}
is periodic in $\lg n$ with period~$1$ and amplitude
smaller than $5 \times 10^{-9}$.
\end{theorem}

Small periodic fluctuations as in \refT{T:fixunif} come as a surprise
to newcomers 
to the analysis of algorithms but in fact are quite common in the analysis of
digital structures and algorithms; see, for example, Chapter~6
in~\cite{MR93f:68045}.

For our further results, it is technically convenient to assume that the
number of keys is no longer fixed at~$n$, but rather Poisson distributed with
mean~$\gl$ and independent of the values of the keys.  (In this paper, 
we shall not deal with the ``de-Poissonization'' that would be needed to transfer
results back to the fixed-$n$ model.)
In obvious notation, the Poissonized version of
\eqref{fixexact1}--\eqref{fixasy} is 
\begin{align}
\EE\,B({\gl})
\label{Poiexact}
 &= 2 \sum_{k = 2}^{\infty} (-1)^k \frac{\gl^k}{k!} \times \frac{1}{(k - 1) k [1 - 2^{- (k -
       1)}]} \\
\label{Poiasy}
 &= \gl (\ln \gl) (\lg \gl) - c_1 \gl \ln \gl + c_2 \gl + \pi_{\gl} \gl 
+O(\log \gl)
\mbox{\ \ as $\gl \to \infty$},
\end{align}
with $\pi_\gl$ as in \eqref{pin}.
The exact formula follows immediately from~\eqref{fixexact1}, and the
asymptotic formula 
is established  in \refS{S:trans} as \refP{P:Poiasy}.
We
will also see (\refP{P:varbound}) that $\Var\,B({\gl}) = O(\gl^2)$,
so~$B({\gl})$ 
is concentrated about its mean.  Since the number $K(\gl)$ of key
comparisons is 
likewise concentrated about its mean $\EE\,K(\gl) \sim 2 \gl \ln \gl$ for
large~$\gl$ (see Lemmas~\ref{L:Poikey} and~\ref{L:Poikeymoments}), it
follows that 
\begin{equation}
\label{bitsperkey}
\frac{2}{\lg \gl} \times \frac{B(\gl)}{K(\gl)} \to 1\mbox{\ \ in probability as
$\gl \to \infty$}.
\end{equation}
In other words, about $\frac{1}{2} \lg \gl$ bits are compared per key
comparison.

\begin{Remark}
Further terms can be obtained in \eqref{fixasy} and \eqref{Poiasy} by the
methods used in the proofs below. 
In particular, the $O(\log\gl)$ in \eqref{Poiasy} can be refined to
\begin{equation*}
  -2 \log\gl - c_4 + O(\gl^{-M})
\end{equation*}
for any fixed $M$, with 
$$
c_4 := 4 \ln 2 + 2 + 2 \gamma \doteq 5.927.
$$
\end{Remark}

For non-uniform distribution~$F$, we have the same leading term for the
asymptotic expansion of $\EE\,B(\gl)$, but the second-order term is larger. 
(Throughout, $\ln_+$ denotes the positive part of the natural logarithm
function.
We denote the uniform distribution by $\unif$.)

\begin{theorem} \label{T:Poif}
Let $X_1, X_2, \dots$ be independent with a common distribution~$F$
over $(0, 1)$ 
having density~$f$, and let~$N$ be independent and Poisson with mean~$\gl$.  If
\mbox{$\int^1_0\!f (\ln_+ f)^4 < \infty$}, then the expected number of bit
comparisons, call it $\mu_f(\gl)$, required to sort the keys $X_1,
\dots, X_N$ using {\tt Quicksort} satisfies
$$
\mu_f(\gl) = \mu_{\unif}(\gl) + 2 H(f) \gl \ln \gl +
o(\gl \log \gl)
$$
as $\gl \to \infty$, where $H(f) := \int^1_0\!f \lg f \geq 0$ is the
entropy (in bits) of the density~$f$.
\end{theorem}

In applications, it may be unrealistic to assume that a specific
density $f$ is known. Nevertheless, even in such cases, \refT{T:Poif}
may be useful since it provides a measure of the robustness of the
asymptotic estimate
in \refT{T:fixunif}.

Bob Sedgewick (among others who heard us speak on the 
material of this paper) suggested that the number of bit comparisons for 
{\tt Quicksort} might be reduced substantially by not comparing bits that have to be
equal according to the results of earlier steps in the algorithm. 
In the final section (Theorem~\ref{T:BQ}),
we note that this is indeed the case: For a
fixed number~$n$ 
of keys, the average number of bit comparisons in the improved algorithm (which
we dub
``{\tt BitsQuick}'')
is asymptotically equivalent to $2 (1 + \frac{3}{2 \ln 2}) n \ln
n$, only a constant ($\doteq 3.2$) times the average number of key comparisons
[see~\eqref{fixkeyasy}]. 
A related algorithm is the digital version of Quicksort by
Roura \cite{Roura}; it too requires $\Theta(n\log n)$ bit comparisons
(we do not know the exact constant factor).

We may compare our results 
to those obtained for 
radix-based methods, for example
radix exchange sorting, 
see \cite[Section 5.2.2]{MR56:4281}.
This method works by bit inspections, that is, by comparisons to constant
bits, rather than by  pairwise comparisons. 
In the case of $n$ uniformly distributed keys, 
radix exchange sorting  uses asymptotically
$n\lg n$ bit inspections. Since radix exchange sorting is designed so
that the number of bit inspections is minimal, it is not surprising
that our results show that 
{\tt Quicksort} uses more bit comparisons.
More precisely, \refT{T:fixunif} shows that 
{\tt Quicksort} uses about $\ln n$ times as many bit comparisons as
radix exchange sorting.
For 
{\tt BitsQuick}, this is reduced to a small constant factor.
This gives us a measure of the cost in bit comparisons
of using these algorithms;
{\tt Quicksort} is often used 
because of other advantages, and
our results
open the possibility of
seeing when they outweigh the increase in bit comparisons.

In \refS{S:fixkey} we review {\tt Quicksort} itself and basic facts about the
number~$K_n$ of key comparisons.  In \refS{S:fixexact1} we derive the exact
formula~\eqref{fixexact1} for $\EE\,B_n$, and in \refS{S:fixasy} we
derive the asymptotic expansion~\eqref{fixasy} from an alternative exact
formula that is somewhat less elementary than~\eqref{fixexact1} but much more
transparent for asymptotics.  In the transitional \refS{S:trans} we
establish certain basic facts about the moments of $K(\gl)$ and $B(\gl)$ in
the Poisson case with uniformly distributed keys, and in \refS{S:Poif}
we use martingale 
arguments to establish \refT{T:Poif} for the expected number of bit comparisons
for Poisson($\gl$) draws from a general density~$f$.  Finally, in
\refS{S:quicker} 
we study the improved
{\tt BitsQuick}
algorithm discussed in the preceding
paragraph. 

\begin{Remark}\rm
The results can be generalized to bases other than 2. For example,
base 256 would give corresponding results on the ``byte complexity''.
\end{Remark}

\begin{Remark}\rm
Cutting off and sorting small subfiles differently would 
affect the results in
Theorems~\ref{T:fixunif} and~\ref{T:Poif}
by $O(n \log n)$ and $O(\gl\log \gl)$ only.
In particular,
the leading terms would remain the same.
\end{Remark}

\begin{Remark}\label{R:lit}\rm
In comparison with the extended abstract~\cite{fjbits2004}, new in this expanded treatment
are \refR{R:error}, Propositions~\ref{P:Poiasy} and~\ref{P:lpbound}, and \refL{L:fstar}, together with
complete proofs of \refT{T:Poif}, Lemmas~\ref{L:Poikey} and~\ref{L:Poikeymoments}, and \refR{R:holder}.
\refS{S:quicker} has been substantially revised.

In the time between~\cite{fjbits2004} and the present paper, the following developments have occurred:
\begin{itemize}
\item Fill and Nakama~\cite{FNexpectedbits} followed the same sort of approach as in this paper to obtain certain
exact and asymptotic expressions for the number of bit comparisons required by {\tt Quickselect}, a close cousin of 
{\tt Quicksort}.
\item Vall\'{e}e et al.~\cite{vcff2009} used analytic-combinatorial methods to extend the results of~\cite{fjbits2004} 
and~\cite{FNexpectedbits} by deriving asymptotic expressions for the expected number of symbol comparisons for both {\tt Quicksort} and {\tt Quickselect}.  In their work, as in the present paper, the keys are assumed to be independent and identically distributed, but the authors allow for quite general probabilistic models (also known as ``sources'') for how each key is generated as a symbol string. 
\item Fill and Nakama~\cite{FNqselectlimdistn} (see also~\cite{n2009}) obtained, for quite general sources, a limiting distribution for the (suitably scale-normalized) number of symbol comparisons required by {\tt Quickselect}.
\item Fill~\cite{fjfeb2010} obtained, for quite general sources, a limiting distribution for the (suitably center-and-scale-normalized) number of symbol comparisons required by {\tt Quicksort}.
\end{itemize}
We were motivated to expand~\cite{fjbits2004} to the present full-length paper in large part because this paper's 
Lemmas~\ref{L:Poikey} and~\ref{L:Poikeymoments}, and an extension of (the proof of) \refP{P:lpbound}, play key roles 
in~\cite{fjfeb2010}.
\end{Remark}

\section{Review:\ number of key comparisons used by Quicksort}
\label{S:fixkey}
In this section we briefly review certain basic known results concerning the
number~$K_n$ of key comparisons required by {\tt Quicksort} for a fixed
number~$n$ of keys uniformly distributed on $(0, 1)$.  (See, for example,
\cite{MR1932675} and the references therein for further details.)

{\tt Quicksort}, invented by Hoare~\cite{MR25:5609}, is the standard sorting
procedure in {\tt Unix} systems, and has been cited~\cite{DS} as one of the ten
algorithms ``with the greatest influence on the development and practice
of science and engineering in the 20th century.''  The {\tt Quicksort}
algorithm 
for sorting an array of~$n$ distinct keys is very simple to describe.  If $n
= 0$ or $n = 1$, there is nothing to do.  If $n \geq 2$, pick a key
uniformly at 
random from the given array and call it the ``pivot''.  Compare the
other keys to 
the pivot to partition the remaining keys into two subarrays.  Then recursively
invoke {\tt Quicksort} on each of the two subarrays.

With $K_0 := 0$ as initial condition, $K_n$ satisfies the distributional
recurrence relation
$$
K_n \Leq K_{U_n - 1} + K^*_{n - U_n} + n - 1,\qquad n \geq 1,
$$
where~$\Leq$ denotes equality in law (i.e.,\ in distribution), and where, on
the right, $U_n$ is distributed uniformly over the set $\{1, \dots, n\}$,
$K_j^* \Leq K_j$ for  every~$j$, and
$$
U_n;\ K_0, \dots, K_{n - 1};\ K^*_0, \dots, K^*_{n - 1}
$$
are all independent.

Passing to expectations we obtain the ``divide-and-conquer'' recurrence
relation
$$
\EE\,K_n = \frac{2}{n} \sum_{j = 0}^{n - 1} \EE\,K_j +n-1,
$$
which is easily solved to  give
\begin{align}
\label{fixkeyexact}
\EE\,K_n &= 2 (n + 1) H_n - 4 n \\
\label{fixkeyasy}
      &= 2 n \ln n - (4 - 2 \gamma) n + 2 \ln n + (2 \gamma + 1) + O(1 / n).
\end{align}
It is also routine to use a recurrence to compute explicitly the exact variance
of~$K_n$.  In particular, the asymptotics are
$$
\Var\,K_n = \sigma^2 n^2 - 2 n \ln n + O(n)
$$
where $\sigma^2 := 7 - \sfrac{2}{3} {\pi}^2 \doteq 0.4203$.  Higher
moments can be handled similarly.  Further, the normalized sequence 
$$
\widehat{K}_n := (K_n - \EE\,K_n) / n, \qquad n \geq 1,
$$
converges in distribution,
with  convergence of moments of each order, 
to~$\widehat{K}$, where the law of~$\widehat{K}$ is
characterized as the unique distribution over the real line with vanishing
mean that satisfies a certain distributional identity; and the moment
generating functions of~$\widehat{K}_n$ converge 
pointwise
to that of~$\widehat{K}$. 

\section{Exact mean number of bit comparisons}
\label{S:fixexact1}
In this section we establish the exact formula~\eqref{fixexact1}, repeated
here for  convenience as~\eqref{fixexact1again}, for the expected number of bit
comparisons required by {\tt Quicksort} for a fixed number~$n$ of keys
uniformly distributed on $(0, 1)$:
\begin{equation}
\label{fixexact1again}
\EE\,B_n = 2 \sum_{k = 2}^n (-1)^k \binom{n}{k} \frac{1}{(k - 1) k [1 - 2^{-
(k - 1)}]}. 
\end{equation}

Let $X_1, \dots, X_n$ denote the keys, and $X_{(1)} < \cdots < X_{(n)}$ their
order statistics.  Consider ranks $1 \leq i < j \leq n$. 
Formula~\eqref{fixexact1again} follows readily from the following three facts,
all either obvious or very well known:
\vspace{.1in}
\begin{itemize}
\item The event $C_{i j} := \{\text{keys }X_{(i)}$ and $X_{(j)}$ are
compared$\}$ and the random vector $(X_{(i)}, X_{(j)})$ are independent.
\item $\PP(C_{i j}) = 2 / (j - i + 1)$.  [Indeed, $C_{i j}$ equals the event
that the first pivot chosen from among $X_{(i)}, \dots, X_{(j)}$ is either
$X_{(i)}$ or $X_{(j)}$.]
\item The joint density $g_{n, i, j}$ of $(X_{(i)}, X_{(j)})$ is given by
\begin{equation}
\label{joint}
g_{n, i, j}(x, y) = \binom n {i - 1, 1, j - i - 1, 1, n - j} \,x^{i - 1} (y - x)^{j - i - 1} (1 - y)^{n - j}.
\end{equation}
\end{itemize} 

Let $b(x, y)$ denote the index of the first bit at which the numbers $x, y \in
(0, 1)$ differ.
(For definiteness we take in this paper the terminating expansion with
infinitely many zeros for dyadic rationals in $[0,1)$, but $1=.111\dots$.)
Then 
\begin{equation}
\begin{split}
\label{ebnpn}
\EE\,B_n &= \sum_{1 \leq i < j \leq n} \PP(C_{i j})
\int^1_0\!\int^1_x\!b(x, y)\,g_{n, i, j}(x, y)\,dy\,dx \\
&=
\int^1_0\!\int^1_x\!b(x, y)\,p_n(x, y)\,dy\,dx,
\end{split}
\end{equation}
where $p_n(x, y)$ has the definition and interpretation
\begin{align*}
p_n(x, y)
& := \sum_{1 \leq i < j \leq n} \PP(C_{i j}) g_{n, i, j}(x, y)\,dy\,dx \\
 &\phantom{:} 
 =\frac{\PP(\mbox{keys in $(x, x + dx)$ and $(y, y + dy)$ are compared})}
{dx\,dy}.
\end{align*}
By a routine calculation,
\begin{equation}
\label{pnsum}
\begin{split}
\hspace{-.01in} p_n(x, y) 
&= \frac{2}{(y - x)^2} \left[ \left( 1 - (y - x)
\right)^n - 1 + n(y - x) \right] \\
&= 2 \sum_{k = 2}^n (-1)^k \binom{n}{k} (y - x)^{k - 2},
  \end{split}
\end{equation}
which depends on~$x$ and~$y$ only through the difference $y - x$. 
Plugging~\eqref{pnsum} into~\eqref{ebnpn}, we find
$$
\EE\,B_n = 2 \sum_{k = 2}^n (-1)^{k} \binom{n}{k} \int^1_0\!\int^1_x\!b(x, y)
(y - x)^{k - 2}\,dy\,dx.
$$
But, 
by routine (if somewhat
lengthy) calculation,
\begin{align*}
\int^1_0\!\int^1_x\!b(x, y) (y - x)^{k - 2}\,dy\,dx
 &=  \sum_{\ell = 0}^{\infty} (\ell + 1) \int\!\!\!\!\int_{0 < x < y < 1:\,b(x,
       y) = \ell + 1} (y - x)^{k - 2}\,dx\,dy 
\\
 &=
     \sum_{\ell = 0}^{\infty} (\ell + 1) 2^{\ell} \int^{2^{-(\ell +
       1)}}_0\!\!\!\int^{2^{-\ell}}_{2^{-(\ell + 1)}}\!(y - x)^{k -
	   2}\,dy\,dx
\\
 & =  \frac{1}{(k - 1) k [1 - 2^{- (k - 1)}]}.
\end{align*}
This now leads immediately to the desired~\eqref{fixexact1again}.

\section{Asymptotic mean number of bit comparisons}
\label{S:fixasy}
Formula~\eqref{fixexact1}, repeated at~\eqref{fixexact1again}, is hardly
suitable for numerical calculations or asymptotic treatment, due to excessive
cancellations in the alternating sum.  Indeed, if (say) $n = 100$, then the
terms (including the factor~$2$, for definiteness) alternate in sign, with
magnitude as large as $10^{25}$, and yet $\EE\,B_n \doteq 2295$.  Fortunately,
there is a  standard complex-analytic technique designed for precisely our situation
(alternating binomial sums), namely, \emph{Rice's method}.  We 
 will not review the idea behind the method here, but rather
refer the reader to (for example) Section~6.4 of~\cite{MR93f:68045}.  Let
$$
h(z) := \frac{2}{(z - 1) z [1 - 2^{- (z - 1)}]}
$$
and let $B(z, w) := \Gamma(z)\Gamma(w)/ \Gamma(z+w)$ denote the
(meromorphic continuation) of the classical beta function.   According to
Rice's method, $\EE\,B_n$ equals the sum of the residues of the function $B(n +
1, - z) h(z)$ at
\begin{itemize}
\item the triple pole at $z = 1$;
\item the simple poles at $z = 1 + i \beta k$, for $k \in \ZZ\setminus\{0\}$;
\item the double pole at $z = 0$.
\end{itemize}
The residues are easily calculated, especially with the aid of such
symbolic-manip\-ulation software as {\tt Mathematica} or {\tt Maple}. 
Corresponding to the above list, the residues equal
\begin{itemize}
\item
$\frac{n}{\ln 2} \left[ H^2_{n - 1} - (4 - \ln 2) H_{n - 1} + \frac{1}{6}
(6 - \ln 2)^2 + H^{(2)}_{n - 1} \right]$;
\item
$\frac{i}{\pi k (-1 - i\beta k)} \Gamma(-1 - i \beta k) \frac{n!}{\Gamma(n - i
\beta k)}$; \vspace{1mm}
\item $-2 (H_n + 2 \ln 2 + 1)$,
\end{itemize}
\vspace{1mm}
where $H^{(r)}_n := \sum_{j = 1}^n j^{-r}$ denotes the $n$th harmonic number
of order~$r$ and $H_n := H^{(1)}_n$.  Summing the residue contributions gives
an alternative exact formula for $\EE\,B_n$, from which the asymptotic
expansion~\eqref{fixasy} (as well as higher-order terms) can be read off
easily using standard asymptotics for $H^{(r)}_n$ and Stirling's formula; we
omit the 
details.

This completes the proof of \refT{T:fixunif}.

\begin{Remark}\rm
\label{R:keyrice}
We can calculate $\EE\,K_n$ in the same fashion (and somewhat more easily), by
replacing the bit-index function~$b$ by the constant function~$1$.  Following
this approach, we obtain first the following analogue
of~\eqref{fixexact1again}: 
$$
\EE\,K_n = 2 \sum_{k = 2}^n (-1)^k \binom{n}{k} \frac{1}{(k - 1) k}. 
$$
Then the residue contributions using Rice's method are
\begin{itemize}
\item $2 n (H_n - 2 - \frac{1}{n})$, at the double pole at $z = 1$;
\item $2 (H_n + 1)$, at the double pole at $z = 0$.
\end{itemize}
Summing the two contributions gives an alternative derivation
of~\eqref{fixkeyexact}.
\end{Remark}

\section{Poissonized model for uniform draws}
\label{S:trans}
As a warm-up for \refS{S:Poif}, we now suppose that the number of keys
(throughout this section still assumed to be uniformly distributed) is Poisson
with mean~$\gl$.  

\subsection{Key comparisons}
\label{S:key}

We begin with a lemma which provides both the analogue of
\eqref{fixkeyexact}--\eqref{fixkeyasy} and two other facts we will need in
\refS{S:Poif}.

\begl \label{L:Poikey}
In the setting of \refT{T:Poif} with~$F$ uniform, the 
expected number of 
key comparisons is a strictly convex function of~$\gl$ given by 
\begin{align*}
\EE\,K(\gl) &= 2 \int^{\gl}_0\!(\gl - y) (e^{-y} - 1 + y)
y^{-2}\,dy.
\end{align*}
Asymptotically,  as $\gl \to \infty$ we have
\begin{equation}
\label{Poikeyasy}
\EE\,K(\gl) 
= 2 \gl \ln \gl - (4 - 2 \gamma) \gl + 2 \ln \gl + 2 \gamma + 2 + 
O(e^{-\gl} \gl^{-2})
\end{equation}				  
and as $\gl \to 0$ we have
\begin{equation}
\label{smallgl} 
\EE\,K(\gl) = \sfrac{1}{2} \gl^2 + O(\gl^3).
\end{equation}
\enl
\noindent
Comparing the $n \to \infty$ expansion~\eqref{fixkeyasy}
with the corresponding expansion for Poisson($\gl$) many keys,
note the difference in constant terms and the much smaller error term in the
Poisson case.

\begin{proof}
To obtain the exact formula, begin with
$$
\EE\,K_n = \int^1_0\!\int^1_x\!p_n(x, y)\,dy\,dx;
$$
\cf~\eqref{ebnpn} and recall \refR{R:keyrice}.  Then multiply both sides by
$e^{-\gl} \gl^n / n!$ and sum, using the middle expression in \eqref{pnsum}; we
omit the simple computation.
Strict convexity then follows from the calculation $\frac{d^2}{d \gl^2}
\EE\,K(\gl) = 2 (e^{-\gl} - 1 + \gl) / \gl^2 > 0$, and
asymptotics as $\gl \to 0$ are trivial:\ $\EE\,K(\gl) = 2 \int^{\gl}_0\!(\gl
- y) [\frac{1}{2} + O(y)]\,dy = \frac{1}{2} \gl^2 + O(\gl^3)$.

To derive the result for $\gl \to \infty$, letting ${\bf 1}[A]$ denote~$1$ if~$A$ holds
and~$0$ otherwise, we observe 
\begin{align*}
\lefteqn{\hspace{-.1in}\sfrac{1}{2} \EE\,K(\gl)} \\ 
  &= \gl \int^{\infty}_0\!\!\left( e^{-y} - 1 + y {\bf 1}[y < 1] \right) y^{-2}\,dy
            - \gl \int^{\infty}_{\gl}\!(e^{-y} - 1) y^{-2}\,dy + \gl \int^{\gl}_1\!y^{-1}\,dy \\ 
  & \qquad  - \int^{\infty}_0 \left( e^{-y} - {\bf 1}[y < 1] \right) y^{-1}\,dy
           + \int^{\infty}_{\gl}\!e^{-y} y^{-1}\,dy + \int^{\gl}_1\!y^{-1}\,dy - \int^{\gl}_0\,dy \\
  &= -\gl (1 - \gamma) + \left[ 1 - \gl \int^{\infty}_{\gl}\!e^{-y} y^{-2}\,dy \right] + \gl \ln \gl \\
  & \qquad + \gamma + \int^{\infty}_{\gl}\!e^{-y} y^{-1}\,dy + \ln \gl - \gl \\
  &= \gl \ln \gl - (2 - \gamma) \gl + \ln \gl + \gamma + 1 + 
O(e^{-\gl} \gl^{-2}),
\end{align*}
as desired.  The calculations
\begin{align}
\label{firstcalc}
\int^{\infty}_0 \left( e^{-y} - {\bf 1}[y < 1] \right) y^{-1}\,dy &= -\gamma, 
\\
\label{secondcalc}
\int^{\infty}_0\!\!\left( e^{-y} - 1 + y {\bf 1}[y < 1] \right) y^{-2}\,dy 
&= - (1 - \gamma),
\\
\label{thirdcalc}
\int^{\infty}_\gl e^{-y} y^{-1}\,dy &= e^{-\gl} \gl^{-1}+O(e^{-\gl} \gl^{-2}),
\\
\label{fourthcalc}
\int^{\infty}_\gl e^{-y} y^{-2}\,dy &= e^{-\gl} \gl^{-2}+O(e^{-\gl} \gl^{-3}),
\end{align}
used at the second and third 
equalities are justified in Appendix~\ref{S:appendixA}.
\end{proof}

\begin{Remark}\label{R:error}\rm
The error term in \eqref{Poikeyasy} can, using \refL{LA2}, be refined to an
asymptotic expansion.  Indeed, for any $M \ge 1$ it can be written as
 \begin{equation*}
   e^{-\gl}\sum_{k=1}^{M - 1} (-1)^{k + 1} k\cdot k!\,\gl^{-k-1}+O(e^{-\gl}\gl^{-M-1}).
 \end{equation*}
\end{Remark}

To handle the number of bit comparisons, we will also need the following bounds
on the moments of $K(\gl)$.  Together with \refL{L:Poikey}, these bounds also
establish concentration of $K(\gl)$ about its mean when~$\gl$ is large.
For real $1 \leq p < \infty$, we let $\| W \|_p := \left( \EE\,|W|^p
\right)^{1 / p}$ denote $L^p$-norm and use $\EE(W;A)$ as shorthand for the
expectation of the product of~$W$ and the indicator of the event~$A$.

\begl \label{L:Poikeymoments}
For every real $p \geq 1$, there exists a constant $c_p < \infty$ such that
\begin{align*}
&\| K(\gl) - \EE\,K(\gl) \|_p \leq c_p \gl && \text{for $\gl \geq 1$},
\\
&\| K(\gl) \|_p \leq c_p \gl^{2 / p}&&\text{for $\gl \leq 1$}.
\end{align*}
In particular, $\Var K(\gl) \le c_2^2\gl^2$ for all $\gl>0$.
\enl

\begin{proof}
We use the notation
of \refT{T:Poif} with~$F$ uniform [so that $K(\gl) = K_N$ with~$N$ distributed Poisson$(\gl)$]
and write $\kappa_n := \EE\,K_n$ for $n \geq 0$.

(a)~The first result is certainly true for 
$\mbox{$\gl \geq 1$}$ bounded away from~$\infty$.
For $\gl \to \infty$ the result can be established by 
Poissonizing standard {\tt Quicksort} moment
calculations, as we now sketch.
(Although the following argument is valid for all $p\ge1$, the reader that
so prefers may assume that $p$ is an even integer.)
We start with
\begin{equation}
\label{twoterms}
\| K(\gl) - \EE\,K(\gl) \|_p \leq \| K_N - \kappa_N \|_p + \| \kappa_N - \EE\,K(\gl) \|_p
\end{equation} 
and proceed to argue that the first term on the right is asymptotically linear in~$\gl$
while the second term is $o(\gl)$.  

To handle the first term, observe that
$$
\| K_N - \kappa_N \|^p_p = \EE |K_N - \kappa_N|^p = \EE\,\EE [|K_N - \kappa_N|^p\,|\,N].
$$
But
$$
\EE [|K_N - \kappa_N|^p\,|\,N = n] = \EE |K_n - \kappa_n|^p;
$$
by the comments at the very end of \refS{S:fixkey} this equals
$(1 + o(1)) \left( \EE\,|\widehat{K}|^p \right) n^p$ as $n \to \infty$
and so can be bounded for all~$n$ by a constant times $n^p$.
Thus one need only observe that $\EE\,N^p = (1 + o(1)) \gl^p$
as $\gl \to \infty$ to complete treatment of the first term on the right
in~\eqref{twoterms}. 

To treat the second term in RHS\eqref{twoterms} as $\gl \to \infty$, 
one can show using~\eqref{fixkeyasy} and~\eqref{Poikeyasy} 
and the normal approximation to the Poisson that
$$
\| \kappa_N - \EE\,K(\gl) \|_p = (1 + o(1))\,2 \| N \ln N - \gl \ln \gl \|_p = (1 + o(1)) 2 \|Z\|_p\,\gl^{1/2} \ln \gl = o(\gl)
$$
where~$Z$ has the standard normal distribution.  We omit the details.
 
(b)~For $\gl \leq 1$ we use
\begin{align*}
\EE\,K^p(\gl) 
&\leq \EE \left[ {\binom N 2}^p; N \geq 2 \right] \leq
\EE\,[N^{2 p}; N \geq 2]
= \gl^2 \sum_{n = 2}^{\infty} e^{-\gl} \frac{\gl^{n - 2}}{n!}
n^{2 p} \leq c^p_p \gl^2, 
\end{align*}
provided~$c_p$ is taken to be at least the
finite value
$ \left[ \sum_{n = 2}^{\infty} (n^{2 p} / n!) \right]^{1 / p}$.
\end{proof}

\subsection{Bit comparisons}
\label{S:bit}

We now turn our attention from~$K(\gl)$ to the more interesting random
variable~$B(\gl)$, the total number of bit comparisons. We 
discuss first asymptotics for the mean $\mu_{\unif}(\gl)$ and then the
variability of $B(\gl)$ about 
the mean.  In our next proposition we will derive the asymptotic
estimate~\eqref{Poiasy} by applying standard 
asymptotic techniques to the exact formula~\eqref{Poiexact}.

\begp
\label{P:Poiasy}
Asymptotically as $\gl \to \infty$, we have
$$
\mu_{\unif}(\gl) = \EE\,B({\gl}) = \gl (\ln \gl) (\lg \gl) - c_1 \gl \ln \gl + c_2 \gl + \pi_{\gl} \gl + O(\log \gl).
$$
\enp

\begin{proof}[Proof (outline)]
Recalling~\eqref{Poiexact} and noting that for $x > 0$ we have
$$
\sum_{k = 2}^{\infty} (-1)^k \frac{x^k}{k! (k - 1) k} = \int^x_0\!\int^w_0\!v^{-2} (e^{-v} - 1 + v)\,dv\,dw =: g(x),
$$
it follows that $\mu(\gl) \equiv \mu_{\unif}(\gl)$ has the harmonic sum form
$$
\mu(\gl) = 2 \sum_{j = 0}^{\infty} 2^j g(2^{-j} \gl),
$$
rendering it amenable to treatment by Mellin transforms, 
see, e.g.,\ \cite{FGD} or \cite{FS}.
Indeed, it follows immediately that the Mellin transform $\mu^*$ of~$\mu$ is given for~$s$ in the fundamental strip
$\{s \in {\mathbb C}:-2 < \Re s < -1\}$ by
$$
\mu^*(s) = 2 g^*(s) \gL(s)
$$
in terms of the Mellin transform $g^*$ of~$g$ and the generalized Dirichlet series
$$
\gL(s) = \sum_{j = 0}^{\infty} 2^{j (s + 1)} = \frac{1}{1 - 2^{s + 1}}.
$$
But it's also easy to check using the integral formula for~$g$ that
$$
g^*(s) = \frac{\Gamma(s)}{(s + 1) s},
$$
and 
so 
$$
\mu^*(s) = \frac{2\Gamma(s)}{(s + 1) s (1 - 2^{s + 1})}.
$$
The desired asymptotic expansion for $\mu(\gl)$ (including the remainder term) can then be read off from the singular behavior of 
$\mu^*(s)$ at its poles located at $s = -1$ (triple pole), $s = -1 - i \beta k$ for $k \in \ZZ \setminus \{0\}$ (simple poles), and $s = 0$ (double pole), paralleling the use of Rice's method for $\EE\,B_n$ in \refS{S:fixasy}.
\end{proof}

In order to move beyond the mean of $B(\gl)$, we define
\begin{align*}
I_{k, j}
 &:= [(j - 1) 2^{-k}, j 2^{-k})
\\
\intertext{to be the $j$th dyadic rational interval of rank~$k$, and consider} 
B_{k}(\gl)
 &:= \mbox{number of comparisons of $(k + 1)$st bits}, \\
B_{k, j}(\gl)
 &:= \text{number of comparisons of $(k + 1)$st bits
between keys in $I_{k, j}$}. 
\end{align*}
Observe that
\begin{equation}
\label{bdecomp}
B(\gl) = \sum_{k = 0}^{\infty} B_k(\gl) 
= \sum_{k = 0}^{\infty} \sum_{j = 1}^{2^k}
B_{k, j}(\gl).
\end{equation}
A simplification provided by our Poissonization is that, for each
fixed~$k$, the 
variables $B_{k, j}(\gl)$ are independent.  Further, the marginal
distribution of 
$B_{k, j}(\gl)$ is simply that of $K(2^{-k} \gl)$.  

\begr
Taking expectations
in~\eqref{bdecomp}, we find
\begin{equation}
\label{Poi_bit_vs_key}
\mu_{\unif}(\gl) = \EE\,B(\gl) = \sum_{k = 0}^{\infty} 2^k\,
\EE\,K(2^{-k} \gl).
\end{equation}
If one is satisfied with a remainder of $O(\gl)$ rather than $O(\log \gl)$,
then~\refP{P:Poiasy} can also be proved by means of~\eqref{Poi_bit_vs_key}.
This is done by splitting the sum $\sum_{k = 0}^{\infty}$ there into $\sum_{k = 0}^{\lf \lg \gl \rf}$
and $\sum_{k = \lf \lg \gl \rf + 1}^{\infty}$ and utilizing~\eqref{Poikeyasy} (to the needed order) for the first sum
and~\eqref{smallgl} [or rather the simpler $\EE\,K(\gl) = O(\gl^2)$ as 
$\gl\to 0$] for the second.  We omit the details.
(See also \refS{S:Poif} where this argument is used in a more general
situation as part of the proof of \refT{T:Poif}.)
\enr

Moreover,
we are now in position to establish the concentration of $B(\gl)$
about $\mu_{\unif}(\gl)$
promised just prior to~\eqref{bitsperkey}.

\begp
\label{P:varbound}
There exists a constant~$c$ such that $\Var\,B(\gl) \leq c^2 \gl^2$ for $0 < \gl < \infty$.
\enp

\begin{proof}
For $0 < \gl < \infty$, we have
by \eqref{bdecomp},
the triangle inequality for $\|\cdot\|_2$,
independence and $B_{k, j}(\gl) \Leq  K(2^{-k} \gl)$,
and
\refL{L:Poikeymoments}, 
with $c := c_2  \sum_{k = 0}^{\infty} 2^{- k / 2}$,
\begin{equation*}
[\Var B(\gl)]^{1 / 2} \leq \sum_{k = 0}^{\infty} [\Var\,B_k(\gl)]^{1 / 2}
 \leq \sum_{k = 0}^{\infty} [2^k \Var\,K(2^{-k} \gl)]^{1 / 2}
 \leq c \gl.
\end{equation*}
\vskip-2.2\baselineskip
\end{proof}
\bigskip

Our next proposition extends the previous one but is limited to $\gl \geq 1$.

\begp
\label{P:lpbound}
For any real $1 \leq p < \infty$, there exists a constant $c'_p < \infty$ such that
$$
\| B(\gl) - \EE\,B(\gl) \|_p \leq c'_p \gl \qquad \text{for $\gl \geq 1$}.
$$
\enp

\begin{proof}
Because $L^p$-norm is nondecreasing in~$p$, we may assume that $p \geq 2$.
The proof again starts with use of the triangle inequality for $\| \cdot
\|_p$: For $0 < \gl < \infty$ we have from~\eqref{bdecomp} that 
\begin{equation}
\label{normbreakdown}
\| B(\gl) - \EE\,B(\gl) \|_p \leq \sum_{k = 0}^{\infty} \| B_k(\gl) - \EE\,B_k(\gl) \|_p.
\end{equation}
Further, 
\begin{equation*}
 B_k(\gl)  - \EE  B_k(\gl) 
=  \sum_{j = 1}^{2^k} \bigbrack{B_{k, j}(\gl)-\EE B_{k, j}(\gl)},
\end{equation*}
where the summands are independent and centered, each with the same
distribution as $K(2^{-k}\gl)-\EE K(2^{-k}\gl)$.
Hence, by Rosenthal's inequality \cite[Theorem~3]{Rosenthal} (see also, e.g.,\ 
\cite[Theorem 3.9.1]{Gut}) and \refL{L:Poikeymoments},
\begin{equation*}
  \begin{split}
\| B_k(\gl)  - \EE  B_k(\gl) \|_p
&\le b_1 \lrpar{2^{k/p} \|{B_{k, j}(\gl)-\EE B_{k, j}(\gl)}\|_p
+2^{k/2}\|\bigpar{B_{k, j}(\gl)-\EE B_{k, j}(\gl)}\|_2}	
\\&
= b_1 2^{k/p} \|{K(2^{-k}\gl)-\EE K(2^{-k}\gl)}\|_p
+b_12^{k/2}\|{K(2^{-k}\gl)-\EE K(2^{-k}\gl)}\|_2
\\&
\le b_1 2^{k/p}c_p \bigpar{2(2^{-k}\gl)^{2/p}+2^{-k}\gl}
+b_12^{k/2}c_2 2^{-k}\gl
\\&
\le b_2 2^{-k/p}\gl^{2/p}+b_3 2^{-k/2}\gl
  \end{split}
\end{equation*}
for some constants $b_1$, $b_2$ and $b_3$ (depending on $p$).
Therefore, by~\eqref{normbreakdown},
\begin{equation*}
\| B(\gl)  - \EE  B(\gl) \|_p  
\le b'_2 \gl^{2/p}+b'_3 \gl
\le(b'_2+b'_3)\gl
\end{equation*}
when $\gl\ge1$.
\end{proof}

\begin{Remark}
For the (rather uninteresting) case $\gl\le1$, the same proof yields
$\| B(\gl)  - \EE  B(\gl) \|_p \le c_p' \gl^{2/p}$ for $p\ge2$. This inequality
actually holds (for some $c'_p$) for all $p\ge1$; the case $1\le p<2$ follows easily from
\eqref{bdecomp} and \refL{L:Poikeymoments}. 
\end{Remark}

\begr
\label{exactorder}
In~\cite{bfnondeg} it is shown (in a more general setting) that
the variables $B_k(\gl)$ are positively correlated, 
from which it is easy to check that
$\Var\,B(\gl) = \Omega(\gl^2)$ for $\gl \geq 1$.  We 
then have $\| B(\gl) - \EE\,B(\gl) \|_p = \Theta(\gl)$ for each
real $2 \leq p < \infty$.  In fact, it is even true that $[B(\gl) - \EE\,B(\gl)] / \gl$ has a nondegenerate 
limiting distribution:\ see~\cite{fjfeb2010}.
\enr

\section{Mean number of bit comparisons for keys drawn from an arbitrary
density~$f$}
\label{S:Poif}
In this section we outline martingale arguments for proving
\refT{T:Poif} for the 
expected number of bit comparisons for Poisson($\gl$) draws from a
rather general 
density~$f$.  (For background on martingales, see any standard
measure-theoretic 
probability text, \eg,~\cite{MR1796326}.)  In addition to
the notation above,
we will use the following:
\begin{align*}
p_{k, j}  &:= \int_{I_{k, j}}\!f, \\
f_{k, j}  &:= \mbox{(average value of~$f$ over $I_{k, j}$)} = 2^k p_{k, j}, \\
f_k(x)     &:= f_{k, j}\mbox{\ \ for all $x \in I_{k, j}$}, \\
f^*(\cdot) &:= \sup_k f_k(\cdot).
\end{align*}
Note for each $k \geq 0$ that $\sum_j p_{k, j} = 1$ and that $f_k:(0, 1) \to [0, \infty)$ 
is the smoothing of~$f$ to the rank-$k$ dyadic rational intervals. 
{}From basic martingale theory we have immediately the following simple but key
observation.

\begl
\label{L:doob}
With $f_{\infty} := f$,
$$
\text{$(f_k)_{0 \leq k \leq \infty}$ is a Doob's 
martingale,}
$$
and $f_k \to f$ almost surely (and in~$L^1$).
\enl

Our
proof of \refT{T:Poif} will also utilize the following technical lemma.

\begl
\label{L:fstar}
If (as assumed in \refT{T:Poif}) the probability density~$f$ on $(0, 1)$ satisfies
\mbox{$\int^1_0\!f (\ln_+ f)^4 < \infty$}, then
\begin{equation}
\label{fstar}
\int^1_0\!f^* (\ln_+ f^*)^3 < \infty.
\end{equation}
\enl

\begin{proof}
This follows readily by applying
one of the standard maximal inequalities for nonnegative submartingales
which asserts that for a nonnegative submartingale $(Y_k)_{1 \leq k < \infty}$
and 
$Y^* := \sup_{1 \leq k < \infty} Y_k$ we have
\begin{equation}\label{doob1}
\EE\,Y^* \leq \frac{e}{e - 1} \left[ 1 + \sup_{1 \leq k < \infty} \EE(Y_k
  \ln_+ Y_k) \right];
\end{equation}
see, e.g.,\ \cite[Theorem 10.9.4]{Gut}.
The process $(Y_k := f_k (\ln_+ f_k)^3)_{1 \leq k < \infty}$ is a submartingale
by \refL{L:doob} and the convexity of the function $x \to x (\ln_+ x)^3$, and for every $1 \leq k < \infty$
we have
$$
\int^1_0\!Y_k \ln_+ Y_k \leq 
4\int^1_0\!f_k (\ln_+ f_k)^4 \leq 4\int^1_0\!f (\ln_+ f)^4 < \infty,
$$
so \eqref{doob1} does indeed give the desired conclusion.
 \end{proof}

Before we begin the proof of \refT{T:Poif} we remark that the asymptotic
inequality $\mu_f(\gl) \geq \mu_{\unif}(\gl)$ observed
there in fact holds for every $0 < \gl < \infty$.  Indeed,
\begin{equation}
\label{f_vs_unif}
\begin{split}
\mu_f(\gl) 
&= \sum_{k = 0}^{\infty} \sum_{j = 1}^{2^k}
 \EE\,K(\gl p_{k, j}) \\
& \geq
\sum_{k = 0}^{\infty} 2^k \EE\,K(\gl 2^{-k}) 
= \mu_{\unif}(\gl),
\end{split}
\end{equation}
where the first equality appropriately generalizes~\eqref{Poi_bit_vs_key}, the
inequality follows by the convexity of $\EE\,K(\gl)$ (recall \refL{L:Poikey}),
and the second equality follows by~\eqref{Poi_bit_vs_key}.
Furthermore, strict inequality $\mu_f(\gl) > \mu_{\unif}(\gl)$ holds
unless $p_{k,j}=2^{-k}$ for all $k$ and $j$, i.e., unless the distribution
$F$ is uniform. (This argument is valid also if $F$ does not have a density.)
\smallskip

\begin{proof}[Proof of Theorem \ref{T:Poif}]
Assume  $\gl \geq 1$ and, with $m \equiv m(\gl) := \lc \lg \gl \rc$, split the
double sum in~\eqref{f_vs_unif} as
\vspace{-.2in}
\begin{equation}
\label{split}
\mu_f(\gl) = \sum_{k = 0}^m \sum_{j = 1}^{2^k} \EE\,K(\gl p_{k, j}) + R(\gl),
\end{equation}
with $R(\gl)$ a remainder term.  
Our first aim is to show that
$$
R(\gl) := \sum_{k = m + 1}^{\infty} \sum_{j = 1}^{2^k}
 \EE\,K(\gl p_{k, j}) = O(\gl).
$$

Since $\EE\,K(\cdot)$ is nondecreasing, we have the inequality
\begin{align*}
\EE\,K(\gl p_{k, j}) 
  &\leq \sum_{n = - \infty}^{\infty} \EE\,K(2^{n + 1})\,{\bf 1}[2^n \leq \gl p_{k, j} < 2^{n + 1}] \\
  &\leq \sum_{n = - \infty}^{\infty} 2^{-n}\,\EE\,K(2^{n + 1})\,\gl p_{k, j}\,{\bf 1}[\gl p_{k, j} \geq 2^n].
\end{align*}
Now if $\gl p_{k, j} \geq 2^n$, then for $x \in I_{k, j}$ we have
$$
f^*(x) \geq f_k(x) = 2^k\,p_{k, j} \geq 2^k \gl^{-1} 2^n \geq 2^{k - m + n}.
$$
Hence
\begin{align*}
\EE\,K(\gl p_{k, j}) 
  &\leq \sum_{n = - \infty}^{\infty} 2^{-n}\,\EE\,K(2^{n + 1})\,\gl p_{k, j}\,{\bf 1}[\gl p_{k, j} \geq 2^n] \\
  &\leq \gl \sum_{n = - \infty}^{\infty} 2^{-n}\,\EE\,K(2^{n + 1})\,\int_{I_{k, j}}\!f_k(x)\,{\bf 1}[2^{k - m + n} \leq f^*(x)]\,dx
\end{align*}
and therefore
\begin{align*}
\sum_{j = 1}^{2^k} \EE\,K(\gl p_{k, j}) 
  &\leq \gl \sum_{n = - \infty}^{\infty} 2^{-n}\,\EE\,K(2^{n + 1})\,\int^1_0\!f_k(x)\,{\bf 1}[2^{k - m + n} \leq f^*(x)]\,dx \\
  &\leq \gl \int^1_0\!f^*(x) \sum_{n = - \infty}^{\infty} 2^{-n}\,\EE\,K(2^{n + 1})\,{\bf 1}[2^{k - m + n} \leq f^*(x)]\,dx.
\end{align*}
From this we conclude
\begin{align*}
R(\gl) 
  &\leq \gl \int^1_0\!f^*(x) \sum_{n = - \infty}^{\infty} 2^{-n}\,\EE\,K(2^{n + 1}) \sum_{k = 1}^{\infty} {\bf 1}[2^{k  + n} \leq f^*(x)]\,dx \\
  &= \gl \int^1_0\!f^*(x) \sum_{k = 1}^{\infty} \sum_{n = - \infty}^{\nu(x, k)} 2^{-n}\,\EE\,K(2^{n + 1})\,dx,
\end{align*}
with $\nu(x, k) := \lf \lg f^*(x) \rf - k$.   We proceed to bound the sum on~$n$ here.  If $\nu \leq 0$, then using the bound of (constant times $\gl^2$) on $\EE\,K(\gl)$ from \refL{L:Poikey} we can bound the sum $\sum_{n  \leq \nu} 2^{-n}\,\EE\,K(2^{n + 1})$ by a constant (say, $b'$) times $2^{\nu}$, while if $\nu > 0$ we can again use the estimates from \refL{L:Poikey} to bound, for some constants $b_1, b_2, b''$ the same sum by
$$
b_1 + \sum_{n = 1}^{\nu} 2^{-n}\,b_2\,(n + 1) 2^{n + 1} \leq b'' \nu^2.
$$
Therefore, for another constant~$b$ we have
\begin{align*}
\sum_{k = 1}^{\infty} \sum_{n = - \infty}^{\nu(x, k)} 2^{-n}\,\EE\,K(2^{n + 1}) 
  &\leq \sum_{k = 1}^{\lf \lg f^*(x) \rf - 1} b'' \nu^2(x, k) + \sum_{k = \lf \lg f^*(x) \rf}^{\infty} b' 2^{\nu(x, k)} \\ 
  &\leq \frac{b''}{3 (\ln 2)^3} [\ln_+ f^*(x)]^3 + 2 b' \leq b \left( 1 + [\ln_+ f^*(x)]^3 \right).
\end{align*}
Using \refL{L:fstar} we finally conclude
$$
R(\gl) \leq b\,\gl \int^1_0\!f^* [1 + (\ln_+ f^*)^3] = O(\gl).
$$ 

Plugging $R(\gl) = O(\gl)$ and the consequence
$$
\EE\,K(x) = 2 x \ln x - (4 - 2 \gamma) x  + O(x^{1 / 2}),
$$
which holds uniformly in $0 \leq x < \infty$,
of \refL{L:Poikey} into~\eqref{split}, we find
\begin{align*}
\mu_f(\gl)
 &=\sum_{k = 0}^m \sum_{j = 1}^{2^k} \Bigl[ 2 \gl p_{k, j} (\ln \gl
 + \ln p_{k, 
       j}) 
- (4 - 2 \gamma) \gl p_{k, j} 
+ O\left( \left( \gl p_{k, j}
       \right)^{1 / 2} \right) \Bigr] + O(\gl) \\
 &= \sum_{k = 0}^m \biggl[ 2 \gl \ln \gl + 2  \gl \sum_{j = 1}^{2^k} p_{k,
       j} \ln p_{k, j} 
- (4 - 2 \gamma) \gl 
+ O\left( \gl^{1
 / 2} 2^{k / 2} \right) \biggr] + O(\gl) \\
 &= \mu_{\unif}(\gl) + 2 \gl \sum_{k = 0}^m \int\!f_k \ln
       f_k + O(\gl),
\end{align*}
where we have used the Cauchy--Schwarz inequality at the second equality
and comparison with the uniform case ($f \equiv 1$) at the third. 

But, by \refL{L:doob}, \eqref{fstar}, and the dominated convergence theorem,
\begin{equation}
\label{dct}
\int\!f_k \ln f_k \longrightarrow \int\!f \ln f\mbox{\ as $k \to \infty$},
\end{equation}
from which follows
\begin{equation*}
  \begin{split}
\mu_f(\gl) 
&= \mu_{\unif}(\gl) + 2 \gl (\lg \gl) \int\!f \ln
f + o(\gl \log \gl) 
\\
&= \mu_{\unif}(\gl) + 2 \gl (\ln \gl)
\int\!f \lg f + o(\gl \log \gl),	
\end{split}
\end{equation*}
as desired. 
\end{proof}

\begr
\label{R:holder}
If we make the stronger assumption that
$$
\mbox{$f$ is \Holder($\alpha$) continuous on $[0, 1]$ for some $\ga > 0$,}
$$
then we can quantify~\eqref{dct} and improve the $o(\gl \log \gl)$ remainder in
the statement of \refT{T:Poif} to $O(\gl)$.  
A  proof is provided in Appendix~\ref{S:appendixB}.
\enr

\section{An improvement:\ BitsQuick}
\label{S:quicker}
Recall the operation of {\tt Quicksort} described in \refS{S:fixkey}.
Suppose that the pivot [call it $x = 0.x(1)\,x(2)\,\dots$] has its first $m_1$ bits $x(1), x(2), \dots, x(m_1)$ all equal to~$0$.  
Then the subarray of keys smaller than~$x$ all have length-$m_1$ prefix consisting of all $0$s as well, 
and it wastes time to compare these known bits when {\tt Quicksort} is called recursively on this subarray.

\begin{table}[b]
\begin{tabular}{|l|}
\hline
{\bf The routine BitsQuick}$(A, m)$ \\
\\ 
{\bf If} $|A| \leq 1$ \\
\tab {\bf Return} $A$ \\
{\bf Else} \\
\tab {\bf Set} $A_- \leftarrow \emptyset$ and $A_+ \leftarrow \emptyset$ \\ 
\tab {\bf Choose} a random pivot key~$x = 0.x(1)\,x(2)\,\dots$ from~$A$ \\
\tab {\bf If} $x(1) = 0$ \\
\tab \tab {\bf Set} $m_1 \leftarrow 1$ \\
\tab \tab {\bf While} $x(m_1 + 1) = 0$ \\
\tab \tab \tab {\bf Set} $m_1 \leftarrow m_1 + 1$ \\
\tab \tab {\bf For} $y \in A$ with $y \neq x$ \\
\tab \tab \tab {\bf If} $y < x$ \\
\tab \tab \tab \tab {\bf Set} $y \leftarrow L^{m_1}(y)$ and then $A_- \leftarrow A_-
                      \cup \{y\}$ \\
\tab \tab \tab {\bf Else} \\
\tab \tab \tab \tab {\bf Set} $A_+ \leftarrow A_+ \cup \{y\}$ \\
\tab \tab {\bf Set} $A_- \leftarrow$ {\tt BitsQuick}$(A_-, m_1)$ and \\
\tab\tab\phantom{{\bf Set}}
$A_+ \leftarrow$ {\tt BitsQuick}$(A_+, 0)$ \\
\tab \tab {\bf Set} $A \leftarrow A_- \parallel \{x\} \parallel A_+$ \\
\tab {\bf Else} \\
\tab \tab {\bf While} $x(m_1 + 1) = 1$ \\
\tab \tab \tab {\bf Set} $m_1 \leftarrow m_1 + 1$ \\
\tab \tab {\bf For} $y \in A$ with $y \neq x$ \\
\tab \tab \tab {\bf If} $y < x$ \\
\tab \tab \tab \tab {\bf Set} $A_- \leftarrow A_- \cup \{y\}$ \\
\tab \tab \tab {\bf Else} \\
\tab \tab \tab \tab {\bf Set} $y \leftarrow L^{m_1}(y)$ and then $A_+ \leftarrow A_+
                      \cup \{y\}$ \\
\tab \tab {\bf Set} $A_- \leftarrow$ {\tt BitsQuick}$(A_-, 0)$ and 
\\
\tab\tab\phantom{{\bf Set}}
$A_+ \leftarrow$ {\tt BitsQuick}$(A_+, m_1)$ \\
\tab \tab {\bf Set} $A \leftarrow A_- \parallel \{x\} \parallel A_+$ \\
\tab {\bf Return} $R^m(A)$ \\  
\hline
\end{tabular}
\end{table}

We call
{\tt BitsQuick}
the obvious recursive algorithm that does away with this waste.  
We give one possible implementation
in the boxed pseudocode, which calls for some explanation. 
The initial call to the routine {\tt BitsQuick}$(A, m)$ is to
{\tt BitsQuick}$(A_0, 0)$, where $A_0$ is the full array to be sorted;
in general, the routine {\tt BitsQuick}$(A, m)$ in essence sorts a subarray~$A$ of $A_0$
in which every element has (and is known to have) the same prefix of length~$m$

There, for $m_1 = 0, 1, \dots$, we use
the notation $L^{m_1}(y)$ for the result of rotating to the left $m_1$ bits the register
containing key~$y$---\ie, replacing 
$y = .y(1)\,y(2)\,\dots$
by $.y(m_1~+~1)\,y(m_1~+~2)\,\dots$.
The input~$m$ indicates how many bits each element of the array~$A$ 
needs to be rotated to the right before the routine terminates,  
and $R^m(A)$ (in the last line of the pseudocode) is the resulting array
after these right-rotations.
The symbol~$\parallel$ denotes concatenation (of sorted arrays).
(We omit minor implementational details, such
as how to do sorting in place and to maintain random ordering for the generated
subarrays, that are the same as for {\tt Quicksort} and very well known.)  The routine
{\tt BitsQuick}$(A, m)$ returns the sorted version of~$A$.  

A related but somewhat more complicated algorithm 
has been considered by 
Roura \cite[Section 5]{Roura}.

The following theorem is the analogue for {\tt BitsQuick} of \refT{T:fixunif}.

\begin{theorem}\label{T:BQ}
If the keys $X_1, \dots, X_n$ are independent and uniformly distributed on
$(0, 1)$, then the number~$Q_n$ of bit comparisons required to sort these
keys using {\tt BitsQuick} has expectation given by the following exact and
asymptotic expressions: 
\begin{align*}
\EE\,Q_n
 &=
\sum_{k = 2}^n (-1)^k \binom{n}{k} k^{-1} \left[ \frac{2 (k -
     2)}{1 - 2^{-k}} - \frac{k - 4}{1 - 2^{- (k - 1)}} \right] 
+ 2 n H_n - 5 n +  2 H_n + 1 \\
 &=
\Bigl( 2 + \frac{3}{\ln 2} \Bigr) n \ln n - \ct_1 n  +
       \tilde{\pi}_n n + O(\log^2 n),
\end{align*}
where, with $\beta := 2 \pi / \ln 2$ as before,
$$
\ct_1 := \frac{7}{\ln 2} + \frac{15}{2} - \Bigpar{\frac{3}{\ln 2}+2}\gamma \doteq 13.9
$$ 
and
$$
\tilde{\pi}_n :=
\frac{1}{\ln 2} \sum_{k \in \ZZ:\,k \neq 0} \frac{3 - i \beta k}{1 + i \beta k}
\Gamma(-1 - i \beta k)\, n^{i \beta k}
$$ 
is periodic in $\lg n$ with
period~$1$ and amplitude
smaller than $2 \times 10^{-7}$.
\end{theorem}

\begin{proof}
We establish only the exact expression; the asymptotic expression can be derived from it
using Rice's method, just as we outlined for $\EE\,B_n$ in \refS{S:fixasy}.  Further, in light
of the exact expression~\eqref{fixexact1} for $\EE\,B_n$, we need
only show that the \emph{expected savings} $\EE\,B_n - \EE\,Q_n$ enjoyed by {\tt BitsQuick} relative
to {\tt Quicksort} is given by the expression
\begin{align}
\label{savings}
\EE\,B_n - \EE\,Q_n 
  &= \sum_{k = 2}^n (-1)^k \binom{n}{k} k^{-1} \left\{ \frac{- 2 (k - 2)}{1 - 2^{-k}} + \frac{(k - 3) (k - 2)}{(k - 1) \left[1 - 2^{ - (k - 1)} \right]} \right\} \\
  &\qquad - (2 n H_n - 5 n + 2 H_n + 1). \nonumber
\end{align}

We use the order-statistics notation $X_{(1)}, \dots, X_{(n)}$ from
\refS{S:fixexact1}.  To derive~\eqref{savings}, we will compute the (random)
total savings for all comparisons with $X_{(i)}$ as pivot, sum over $i = 1,
\dots, n$, and take the expectation.  
For convenience, we may assume that the algorithm chooses a pivot also in
the case of a (sub)array with exactly 1 element, although it is not
compared to anything; thus every key becomes a pivot.
Observe that $X_{(i)}$ is compared as
pivot with keys $X_{(L)}, \dots, X_{(R)}$ 
(except itself)
and with no others, where $L
\equiv L(i)$ and $R \equiv R(i)$ with $L \leq i \leq R$ are the (random)
values uniquely determined by the condition that $X_{(i)}$ is the first
pivot chosen from among $X_{(L)}, \dots, X_{(R)}$ but not (if $L \neq 1$)
the first from among  
$X_{(L - 1)}, \dots, X_{(R)}$  nor (if $R \neq n$) the first from among
$X_{(L)}, \dots, X_{(R + 1)}$.  
Hence, $X_{(i)}$ is compared as a pivot with $R-L$ other keys.
The comparisons with $X_{(i)}$ as pivot are
performed with the knowledge that all the keys $X_{(L)}, \dots, X_{(R)}$
have values in the interval $(X_{(L - 1)}, X_{(R + 1)})$, where if $L = 1$
we interpret $X_{(0)}$ as $0 = .000\ldots$ and if $R = n$ we interpret
$X_{(n + 1)}$ as $1 = .111\ldots$.  The total savings gained by this
knowledge is  
$\sum_{j \in [L, R]:\,j \neq i} [b(X_{(L - 1)}, X_{(R + 1)}) -
  1] = (R - L)\,[b(X_{(L - 1)}, X_{(R + 1)}) - 1]$, where we recall that
$b(x, y)$ denotes the index of the first bit at which~$x$ and~$y$ differ. 

Therefore the grand total savings is
\begin{align*}
B_n - Q_n
  &= \sum_{i = 1}^n [R(i) - L(i)] \left[ b\left(X_{(L(i) - 1)}, X_{(R(i) + 1)}\right) - 1 \right] \\
  &= \sum_{(l, r):\,1 \leq l \leq r \leq n} (r - l)\,[b(X_{(l - 1)}, X_{(r + 1)}) - 1]\,\Bigl| \{i:(L(i), R(i)) = (l, r)\} \Bigr|,
\end{align*}
and so by independence we have
$$
\EE\,B_n - \EE\,Q_n = \hspace{-.2in}\sum_{(l, r):\,1 \leq l \leq r \leq n} (r - l)\,[\EE\,b(X_{(l - 1)}, X_{(r + 1)}) - 1]\,\EE \Bigl| \{i:(L(i), R(i)) = (l, r)\} \Bigr|.
$$
The second expectation on the right is easily computed:
$$
\EE \Bigl| \{i:(L(i), R(i)) = (l, r)\} \Bigr| = \sum_{i = l}^r \PP[(L(i), R(i)) = (l, r)] = (r - l + 1) \theta(l, r)
$$
where, abbreviating $r - l$ to~$d$ and writing ``xor'' for ``exclusive or'',
$$
\theta(l, r) = 
\begin{cases}
(d + 1)^{-1} - 2 (d + 2)^{-1} + (d + 3)^{-1} & \mbox{if $l \neq 1$ and $r \neq n$} \\ 
(d+ 1)^{-1} - (d + 2)^{-1} & \mbox{if $l = 1$ xor $r = n$} \\ 
(d + 1)^{-1} & \mbox{if $l = 1$ and $r = n$},
\end{cases}
$$
so that
$$
\EE \Bigl| \{i:(L(i), R(i)) = (l, r)\} \Bigr| =
\begin{cases}
2 [(d + 2) (d + 3)]^{-1} & \mbox{if $l \neq 1$ and $r \neq n$} \\ 
(d + 2)^{-1} & \mbox{if $l = 1$ xor $r = n$} \\ 
1 & \mbox{if $l = 1$ and $r = n$},
\end{cases}
$$
Therefore
\begin{align}
\lefteqn{\EE\,B_n - \EE\,Q_n} \nonumber \\
  &= 2 \sum_{(l, r):\,2 \leq l \leq r \leq n - 1} \frac{r - l}{(r - l + 2) (r - l + 3)}\,[\EE\,b(X_{(l - 1)}, X_{(r + 1)}) - 1] \nonumber \\
  &\qquad + \sum_{r = 1}^{n - 1} \frac{r - 1}{r + 1}\,[\EE\,b(0, X_{(r + 1)}) - 1] + \sum_{l = 2}^n \frac{n - l}{n - l + 2}\,[\EE\,b(X_{(l - 1)}, 1) - 1] \nonumber \\
  &\qquad + (n - 1)\,[\EE\,b(0, 1) - 1] \nonumber \\
  &= 2 \sum_{(l, r):\,2 \leq l \leq r \leq n - 1} \frac{r - l}{(r - l + 2) (r - l + 3)}\,[\EE\,b(X_{(l - 1)}, X_{(r + 1)}) - 1] \nonumber \\
  &\qquad + 2 \sum_{r = 1}^{n - 1} \frac{r - 1}{r + 1}\,[\EE\,b(0, X_{(r + 1)}) - 1] \nonumber \\
  &= 2 \sum_{i = 1}^n \sum_{j = i + 2}^n \frac{j - i - 2}{(j - i) (j - i + 1)}\,\EE\,b(X_{(i)}, X_{(j)}) \nonumber \\
  &\qquad + 2 \sum_{j = 2}^n \frac{j - 2}{j}\,\EE\,b(0, X_{(j)}) - q_n \nonumber \\
  &= 2 D_n + 2 E_n - q_n, \label{ABq}
\end{align}
where:\ at the second equality we have used symmetry and the observation that $b(0, 1) = 1$; the last two sums are denoted~$D_n$ and~$E_n$, respectively; and
\begin{align}
q_n 
  &:= 2 \sum_{(l, r):\,2 \leq l < r \leq n - 1} \frac{r - l}{(r - l + 2) (r - l + 3)} + 2 \sum_{r = 2}^{n - 1} \frac{r - 1}{r + 1} \nonumber \\
  &\phantom{:}= 2 n H_n - 5 n + 2 H_n + 1. \label{q}
\end{align}

The expectation $\EE\,b(X_{(i)}, X_{(j)})$ may be computed (for $1 \leq i < j \leq n$) by recalling the joint density $g_{n, i, j}$ of $(X_{(i)}, X_{(j)})$ given at~\eqref{joint}.  We then find
\begin{align*}
\EE\,b(X_{(i)}, X_{(j)})
  &= \sum_{\ell = 0}^{\infty} \PP[b(X_{(i)}, X_{(j)}) \geq \ell + 1] \\
  &= \sum_{\ell = 0}^{\infty} \sum_{m = 1}^{2^{\ell}} \int\!\!\!\!\int_{(m - 1) 2^{- \ell} < x < y < m 2^{- \ell}}\!g_{n, i, j}(x, y)\,dx\,dy \\
  &= \sum_{\ell = 0}^{\infty} \sum_{m = 1}^{2^{\ell}} 
    \int\!\!\!\!\int_{(m - 1) 2^{- \ell} < x < y < m 2^{- \ell}}\!\binom n {i - 1, 1, j - i - 1, 1, n - j} \\
  &\qquad \times \,x^{i - 1} (y - x)^{j - i - 1} (1 - y)^{n - j}\,dx\,dy. 
\end{align*}
Now, suppressing 
some computational details,
\begin{align*}
\lefteqn{\sum_{i = 1}^n \sum_{j = i + 2}^n \frac{j - i - 2}{(j - i) (j - i + 1)}\,\binom n {i - 1, 1, j - i - 1, 1, n - j} \,x^{i - 1} (y - x)^{j - i - 1} (1 - y)^{n - j}} \\
&= \sum_{i = 1}^n \sum_{j = i + 2}^n (j - i - 2)\,\binom n {i - 1, j - i + 1, n - j} \,x^{i - 1} (y - x)^{j - i - 1} (1 - y)^{n - j} \\
&= \sum_{k = 3}^n (k - 3)\,\binom{n}{k} (y - x)^{k - 2} \sum_{i = 0}^{n - k} \binom {n - k}{i} \,x^i (1 - y)^{n - k - i} \\
&= \sum_{k = 3}^n (k - 3)\,\binom{n}{k} (y - x)^{k - 2} [1 - (y - x)]^{n - k} \\
&= \frac{1}{2} \sum_{k = 3}^n (-1)^k (k - 3) (k - 2) \binom{n}{k} (y - x)^{k - 2},
\end{align*}
and so
\begin{align}
D_n 
  &= \sum_{\ell = 0}^{\infty} \sum_{m = 1}^{2^{\ell}} 
         \int\!\!\!\!\int_{(m - 1) 2^{- \ell} < x < y < m 2^{- \ell}} \nonumber \\
  &\hspace{1in} \left[ \frac{1}{2} \sum_{k = 3}^n (-1)^k (k - 3) (k - 2) \binom{n}{k} (y - x)^{k - 2} \right] dx\,dy \nonumber \\
  &= \frac{1}{2} \sum_{\ell = 0}^{\infty} 2^{\ell} 
          \int\!\!\!\!\int_{0 < x < y < 2^{- \ell}} \left[ \sum_{k = 3}^n (-1)^k (k - 3) (k - 2) \binom{n}{k} (y - x)^{k - 2} \right] dx\,dy \nonumber \\
  &= \frac{1}{2} \sum_{k = 3}^n (-1)^k \frac{(k - 3) (k - 2)}{(k - 1) k} \binom{n}{k} \sum_{\ell = 0}^{\infty} 2^{- \ell (k - 1)} \nonumber \\
  &= \frac{1}{2} \sum_{k = 2}^n (-1)^k \binom{n}{k} k^{-1} \frac{(k - 3) (k - 2)}{(k - 1) [1 - 2^{- (k - 1)}]}. \label{A}
\end{align}

Similarly (and somewhat more easily), one sees (for $1 \leq j \leq n$) that
\begin{align*}
\EE\,b(0, X_{(j)})
  &= \sum_{\ell = 0}^{\infty} \PP[b(0, X_{(j)}) \geq \ell + 1] \\
  &= \sum_{\ell = 0}^{\infty} \int^{2^{- \ell}}_0\!\binom n {j - 1, 1, n - j} y^{j - 1} (1 - y)^{n - j}\,dy
\end{align*}
and that
\begin{align*}
\sum_{j = 2}^n \frac{j - 2}{j}\,\binom n {j - 1, 1, n - j} y^{j - 1} (1 - y)^{n - j}
&= \sum_{k = 2}^n (-1)^{k - 1} (k - 2) \binom{n}{k} y^{k - 1},
\end{align*}
whence
\begin{align}
E_n 
  &=\sum_{\ell = 0}^{\infty} \int^{2^{- \ell}}_0\!
         \left[ \sum_{k = 2}^n (-1)^{k - 1} (k - 2) \binom{n}{k} y^{k - 1} \right] dy \nonumber \\
  &= \sum_{k = 2}^n (-1)^{k - 1} \frac{k - 2}{k} \binom{n}{k} \sum_{\ell = 0}^{\infty} 2^{- \ell k} \nonumber \\
  &= \sum_{k = 2}^n (-1)^{k - 1} \binom{n}{k} k^{-1} \frac{k - 2}{1 - 2^{- k}}. \label{B}
\end{align}
Plugging~\eqref{q}--\eqref{B} into~\eqref{ABq}, we obtain~\eqref{savings}, thus completing the proof.
\end{proof}

\appendix
\section{Some calculus}
\label{S:appendixA}

The following calculus lemmas
establish the calculations \eqref{firstcalc}--\eqref{fourthcalc}
used in the proof of \refL{L:Poikey}.

\begl
Define
\begin{align*}
\gamma_0(z) &:= \int^{\infty}_0\!e^{-y} y^z\,dy, && \Re z > -1; \\
\gamma_1(z) &:= \int^{\infty}_0\!\left( e^{-y} - {\bf 1}[y < 1] \right) y^z\,dy, && \Re z > -2; \\
\gamma_2(z) &:= \int^{\infty}_0\!\left( e^{-y} - 1 + y {\bf 1}[y < 1] \right) y^z\,dy, && -3 < \Re z < -1.
\end{align*}
Then the following identities hold for $z \neq -1$:
\begin{align*}
\gamma_0(z) &= \Gamma(z + 1), \\
\gamma_1(z) &= (z + 1)^{-1} [\gamma_0(z + 1) - 1] = (z + 1)^{-1} [\Gamma(z + 2) - 1], \\
\gamma_2(z) &= (z + 1)^{-1} [1 + \gamma_1(z + 1)],
\end{align*}
and so $\gamma_1(-1) = \Gamma'(1) = - \gamma$ and $\gamma_2(-2) = - [1 + \gamma_1(-1)] = - (1 - \gamma)$.
\enl

\begin{proof}
The identity for $\gamma_0$ is the definition of the function~$\Gamma$, and the identities for $\gamma_1$ and $\gamma_2$
follow by integration by parts.  Since $\gamma_1(z)$ is continuous in~$z$ for $\Re z > -2$, it follows from the identity for $\gamma_1(z)$
by passage to the limit that $\gamma_1(-1) = \Gamma'(1) = - \gamma$.  Finally, we obtain the desired value of $\gamma_2(-2)$ simply
by plugging $z = -2$ into the identity for $\gamma_2(z)$.
\end{proof}

Let $s\fall k$ denote the falling factorial power $s(s-1)\dotsm(s-k+1)$.

\begin{lemma}
  \label{LA2}
For any fixed $s\in\bbC$ and $M=0,1,\dots$, and all $\gl\ge1$,
\begin{equation*}
  \int^{\infty}_\gl\!e^{-y} y^{s}\,dy 
= e^{-\gl} \gl^{s}\lrbrack{\sum_{k=0}^{M-1} s\fall k \gl^{-k}
 +O(\gl^{-M})}.
\end{equation*}
(The implicit constant depends on $s$ and $M$, but not on $\gl$.)
\end{lemma}

\begin{proof}
For $\gl > 0$, let $\ils :=  \int^{\infty}_\gl\!e^{-y} y^{s}\,dy$.
If $\Re s\le0$, then
\begin{equation*}
  |\ils|\le \int^{\infty}_\gl e^{-y} y^{\Re s}\,dy 
\le \int^{\infty}_\gl e^{-y} \gl^{\Re s}\,dy 
=\gl^{\Re s}e^{-\gl} ,
\end{equation*}
which yields the result for $\Re s\le M=0$.

Further, integration by parts yields 
\begin{equation*}
  \ils=e^{-\gl}\gl^s+s\ilsi,
\end{equation*}
and the result for $\Re s\le M$ follows by induction on~$M$.
Finally, if $\Re s > M$, we use the result just proven with $M$ replaced by
some $M'\ge\Re s$.
\end{proof}

\section{Proof of \refR{R:holder}}
\label{S:appendixB}

We prove that if
\begin{equation}
\label{holdass}
\mbox{$f$ is \Holder($\alpha$) continuous on $[0, 1]$ for some $\ga > 0$,}
\end{equation}
then, as claimed in~\refR{R:holder}, the conclusion of \refT{T:Poif} holds with the remainder $o(\gl \log \gl)$ improved to $O(\gl)$.

\begin{proof}
Using the notation $m \equiv m(\gl) := \lc \lg \gl \rc$ of the proof of \refT{T:Poif} appearing in \refS{S:Poif}, it follows from that proof that we need only establish the asymptotic estimate
$$
\sum_{k = 0}^m \left( \int\!f_k \ln f_k - \int\!f \ln f \right) = O(1)
$$
as $\gl \to \infty$, and for this it is clearly sufficient to show that
\begin{equation}
\label{flogf}
f_k(x) \ln f_k(x) - f(x) \ln f(x) = O((k + 1) 2^{- k \alpha})\mbox{\ uniformly in $x \in [0, 1]$}.
\end{equation}
But indeed~\eqref{holdass} evidently implies
$$
f_k(x) - f(x) = O( 2^{- k \alpha})\mbox{\ uniformly in $x \in [0, 1]$},
$$
and thence, routinely, \eqref{flogf}.
\end{proof}

\end{document}